\documentclass{amsart}
\usepackage{amsmath}
\newtheorem{theorem}{Theorem}[section]
\newtheorem{lemma}[theorem]{Lemma}

\theoremstyle{definition}
\newtheorem{definition}[theorem]{Definition} 

\newtheorem{constr}[theorem]{Construction}

\newcommand{\fF}{{\mathfrak F}} 

\newcommand{\Z}{\mathcal{Z}}

\newcommand{\ideq}{\unlhd}
\newcommand{\sn}{\mbox{$\triangleleft\mspace{-1.8mu}\triangleleft\medspace$}}
\DeclareMathOperator{\Leib}{Leib}
\DeclareMathOperator{\ch}{char}
\DeclareMathOperator{\Ann}{Ann}

\title{Faithful representations of Leibniz algebras}
\author{Donald W. Barnes}
\address{1 Little Wonga Rd.\\Cremorne NSW 2090\\Australia\\}
\email{donwb@iprimus.com.au}


\begin{document}

\begin{abstract} Let $L$ be a Leibniz algebra of dimension $n$.  I prove  the existence of a faithful $L$-module of dimension less than or  equal to $n+1$.
\end{abstract}

\maketitle

\section{Introduction}

A left Leibniz algebra is a linear algebra $L$ whose left multiplication operators $d_a: L \to L$, defined by $d_a(x) = ax$ for all $a, x \in L$, are derivations.  A module for the Leibniz algebra $L$ is a vector space $V$ with two bilinear compositions $xv, vx$ for $x \in L$ and $v \in V$ such that 
\begin{equation*}\begin{split}
x(yv) &= (xy)v + y(xv)\\
x(vy) &= (xv)y + v(xy)\\
v(xy) &= (vx)y + x(vy)\\ \end{split} \end{equation*}
for all $x,y \in L$ and $v \in V$. 
These are precisely the conditions which would be satisfied if $L$ and $V$ were contained as subalgebra and abelian ideal in some Leibniz algebra.  Denoting the left action $v \mapsto xv$ by $\lambda_x$ and the right action $v \mapsto vx$ by $\rho_x$, the above conditions are equivalent to
\begin{equation*}\begin{split}
\lambda_x   \lambda_y &= \lambda_{xy} + \lambda_y   \lambda_x\\
\lambda_x \rho_y &=\rho_y \lambda_x + \rho_{xy}\\
\rho_{xy} &= \rho_y \rho_x + \lambda_x\rho_y\\
\end{split}\end{equation*}
The basic theory of Leibniz algebras and their modules is set out in Loday and Pirashvili \cite{LP}.  For any $L$-module $V$, I denote by $\ker(V)$ the kernel of the representation of $L$ on $V$ in $L$, that is, $\{x \in L \mid \lambda_x = \rho_x = 0\}$.  I denote the centre of $L$ by $\Z(L)$ and the left annihilator $\{x\in L\mid xL=0\}$ by $\Ann(L)$.

The subspace $\langle x^2 \mid x \in L \rangle$ spanned by the squares of elements of $L$ is called the Leibniz kernel of $L$ and denoted $\Leib(L)$.  It is an abelian ideal of $L$, $\Leib(L)L = 0$ and $L/\Leib(L)$ is a Lie algebra.  Thus Leibniz algebras are almost Lie algebras and it is natural to consider which of the theorems about Lie algebras generalise to Leibniz algebras.  For example, Ayupov and Omirov \cite[Theorem 2]{AyO} and Patsourakos \cite[Theorem 7]{Pats} have shown that Engel's Theorem holds for Leibniz algebras. 

For a finite-dimensional Lie algebra $L$ over a field $F$, the Ado-Iwasawa Theorem asserts that there exists a faithful finite-dimensional $L$-module $V$.  There are several extension of this result which assert the existence of such a module $V$ with additional properties.  For example, Hochschild \cite{Hoch} proved that there exists a module $V$ with the additional property that the action $\lambda_x$ is nilpotent for all $x \in L$ for which $d_x$ is nilpotent.  Jacobson has proved \cite[Theorem VI.2, p. 205]{Jac} that, if  $\ch(F) = p \ne 0$, then there exists a completely reducible such $V$. Barnes \cite[Theorem 5.1]{extras} has shown the existence of a module $V$ which, in addition to the Hochschild property, has the property that $V$ is $(U,\fF)$-hypercentral for every saturated formation  $\fF$ of soluble Lie algebras and every subnormal subalgebra $U$ of $L$ which is in $\fF$.  I investigate the extent to which these results generalise to Leibniz algebras.  The theory of formations and $\fF$-hypercentral modules for Lie algebras is set out in Barnes and Gastineau-Hills \cite{BGH} and for Leibniz algebras in Barnes \cite{SchunckLeib}. 

By a theorem of Loday and Pirashvili \cite{LP}, if $V$ is an irreducible module for the Leibniz algebra $L$, then $L/ \ker(V)$ is a Lie algebra and $V$ is either symmetric (that is, $vx=-xv$ for all $x \in L$ and $v \in V$) or antisymmetric (that is, $VL = 0$).  Thus $\Leib(L) \subseteq \ker(V)$ if $V$ is completely reducible.  Hence the Jacobson addition to the Ado-Iwasawa Theorem cannot generalise to Leibniz algebras.

An easy consequence of the definition of a Leibniz module (see Patsourakos \cite[Lemma 6]{Pats} is that $\rho_x^n = (-1)^{n-1} \rho_x\lambda_x^{n-1}$ for all $x \in L$.  If $\lambda_x$ is nilpotent, then so is $\rho_x$.  

If $U$ is an ideal of $L$, I write $U \ideq L$.  A subalgebra $U$ is subnormal in $L$ if there exists a chain of subalgebras
$U = U_0 \ideq U_1 \ideq \dots \ideq U_n = L$.  In this case, I write $U \sn L$.

For finite-dimensional Lie algebras, Iwasawa proved the existence of faithful finite-dimensional modules  by proving the existence of finite-dimensional splitting algebras.  In Section \ref{use}, I give the definition of  splitting algebras and explain their use.  In Section \ref{constr}, I give a simple construction for a splitting algebra for a Leibniz algebra $L$ and abelian ideal $A$, but only for the case $AL=0$ needed for the proof of the existence of finite-dimensional faithful modules.  The existence of finite-dimensional splitting algebras for arbitrary $A$ remains an open question.  In Section \ref{main}, I prove the existence of faithful finite-dimensional modules and look at additional conditions we may require of those modules.

The proof of the main result does not depend on the Ado-Iwasawa Theorem.  As Lie algebras are special cases of Leibniz algebras, it might appear that this provides a simple proof of the Ado-Iwasawa Theorem.  It does not.  It proves for a finite-dimensional Lie algebra the existence of a faithful finite-dimensional {\em Leibniz} representation.  The right action on the module may differ from the left action which need not be faithful.

\section{The use of splitting algebras} \label{use}
\begin{definition} Let $A$ be an abelian ideal of the Lie or Leibniz algebra $L$.  A splitting algebra for $(L,A)$ is a pair $(M,B)$ where $M$ is a Lie (respectively Leibniz) algebra containing $L$ and $B$ is an abelian ideal of $M$ such that  $L +B = M$, $L \cap B = A$ and such that $M$ splits over $B$.
\end{definition}

A key step in Iwasawa's proof in \cite{Iw} was his proof of the following theorem.

\begin{theorem} \label{IwSplit} Let $A$ be an abelian ideal of the finite-dimensional Lie algebra $L$ over any field $F$.  Then there exists a finite-dimensional splitting algebra for $(L,A)$.
\end{theorem}

To prove the existence of a faithful finite-dimensional module for the Lie algebra $L$, Iwasawa used induction over the dimension to reduce to the case in which the centre $Z =\langle z \rangle$ is the only minimal ideal of $L$.  With $(M, B)$ a splitting algebra for $(L,Z)$  and $U$ a complement to $B$ in $M$, he constructed a module $V = \langle e, B\rangle $ with the action given by $(u+b)e=b$ and $(u+b)b' = ub'$ for $u \in U$ and $b,b' \in B$.  As $L$ is a subalgebra of $M$, this is also an $L$-module.  Since $z$ acts non-trivially, the unique minimal ideal of $L$ cannot be in the kernel.  Thus $V$ is a faithful $L$-module.

Now suppose that $L$ is a finite-dimensional Leibniz algebra.  I adjust Iwasawa's construction to allow for the fact that the left action of any element of $\Leib(L)$ on any module must be $0$. 
 
\begin{constr} \label{constmod} Let $A$ be an ideal of the Leibniz algebra $L$ such that $AL=0$.  Let $(M,B)$ be a splitting algebra for $(L,A)$ and let $U$ be a complement to $B$ in $M$. Put $V = \langle e, B \rangle $ with the action of $M$ on $V$ given by $\lambda_{u+b}e = 0$, $\rho_{u+b}e= b$, $\lambda_{u+b}b' = ub'$ and $\rho_{u+b}b' = 0$ for $u \in U$ and $ b,b' \in B$. 
\end{constr}

I shall call this module $V$  the Construction \ref{constmod} module for $(M,B)$.  Provided that $A \supseteq \Ann(L)$, the left action of $U \simeq L/A$ on $B$ is faithful, so $\ker_M(V) \subseteq B$. But the right action of $B$ on $V$ is faithful.  It follows that $\ker_M(V)=0$.  Thus $M$ and $L$ are faithfully represented on $V$.

\section{Construction of a splitting algebra} \label{constr}
I now give a construction for a splitting algebra under the conditions required for the main theorem.

\begin{constr} \label{constsplit}
Let $A$ be an abelian ideal of $L$ with $AL=0$.    Let $W$ be an $L$-module  with $\pi:L \to W$ a left $L$-module isomorphism and with $WL=0$.  Put $A_1 = \pi(A)$.  Construct $X$ the split extension of $W$ by $L$. In $X$, form $D = \{a-\pi(a) \mid a \in A\}$ and  $E = \{x-\pi(x)  \mid x \in L\}$.  The diagrams below show the relationships of the spaces considered.
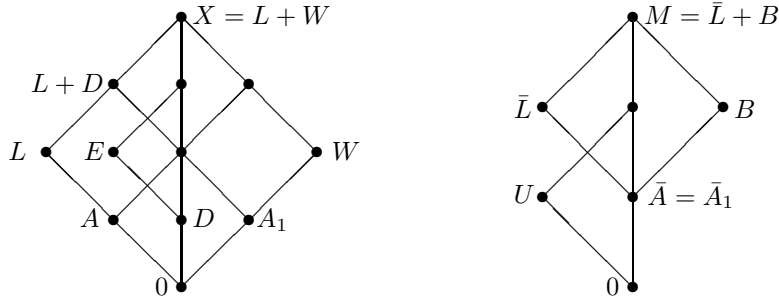
\begin{center}
\setlength{\unitlength}{1mm}
\begin{picture}(120,40)(-54,0)
\put(-20,1){\circle*{1.5}}     \put(-23.5,0){$0$}
\put(-20,1){\line(-1,1){18}}
\put(-38,19){\circle*{1.5}}     \put(-43,18){$L$}
\put(-29,10){\circle*{1.5}}     \put(-33.5,9){$A$}
\put(-20,1){\line(1,1){18}}
\put(-38,19){\line(1,1){18}}
\put(-29,10){\line(1,1){18}}
\put(-11,10){\circle*{1.5}}     \put(-10,9){$A_1$}
\put(-11,10){\line(-1,1){18}}
\put(-2,19){\circle*{1.5}}       \put(0,18){$W$}
\put(-2,19){\line(-1,1){18}}
\put(-20,1){\line(0,1){36}}
\put(-20,10){\line(-1,1){9}}
\put(-29,19){\circle*{1.5}}     \put(-33,18){$E$}
\put(-29,19){\line(1,1){9}}
\put(-20,10){\circle*{1.5}}     \put(-18.5,9){$D$}
\put(-20,28){\circle*{1.5}}
\put(-20,19){\circle*{1.5}}
\put(-20,37){\circle*{1.5}}     \put(-18.5,36){$X=L+W$}
\put(-29,28){\circle*{1.5}}     \put(-40,27){$L+D$}
\put(-11,28){\circle*{1.5}}

\put(40,1){\circle*{1.5}}    \put(36.5,0){$0$}
\put(40,1){\line(0,1){36}}
\put(40,13){\circle*{1.5}}    \put(42,12){$\bar{A}=\bar{A}_1$}
\put(40,13){\line(-1,1){12}}
\put(28,25){\circle*{1.5}}    \put(24.2,23.5){$\bar{L}$}
\put(28,25){\line(1,1){12}}
\put(40,37){\circle*{1.5}}    \put(41.5,36){$M=\bar{L}+B$}
\put(40,13){\line(1,1){12}}
\put(52,25){\circle*{1.5}}    \put(53.5,23.5){$B$}
\put(52,25){\line(-1,1){12}}

\put(40,1){\line(-1,1){12}}
\put(28,13){\circle*{1.5}}    \put(24.2,12.2){$U$}
\put(28,13){\line(1,1){12}}
\put(40,25){\circle*{1.5}}
\end{picture}
\end{center}

$D$ is an ideal of the algebra $X$.  We form the quotient $M = X/D$ and put $B = (W+D)/D$.   We have $\bar{L} = (L+D)/D \simeq L$ with the ideal $\bar{A}$ corresponding to $A$.  The ideal $B$ of $M$ satisfies $BM = 0$, $\bar{L}+B = M$ and $B\cap \bar{L} = \bar{A}$.  For elements $e_1=x_1-\pi(x_1)$ and $e_2=x_2-\pi(x_2)$ in $E$, we have $e_1e_2 = x_1x_2-x_1\pi(x_2) = x_1x_2-\pi(x_1x_2)$ since $WL=0$ and $\pi$ is a left module isomorphism.  Thus $E$ is a subalgebra of $X$.  It follows that $U= E/D$ is a subalgebra of $M$ complementing $B$ and that $(M, B)$ is a splitting algebra for $(\bar{L}, \bar{A})$ and so, for $(L,A)$.  
\end{constr}

I shall refer to this splitting algebra $(M,B)$ as the Construction \ref{constsplit} splitting algebra for $(L,A)$.  Note that $\dim(B) = \dim(L)$.

\section{The main theorem} \label{main}

Now consider the Hochschild addition to the Ado-Iwasawa Theorem which requires that the left action $\lambda_x$ of $x$ on $V$ be nilpotent for all $x \in L$ for which the left action $d_x$ of $x$ on $L$ is nilpotent.  To incorporate this condition into the theorem for Leibniz algebras, I consider how elements of $L$ with $d_x$ nilpotent act on our Construction \ref{constsplit} splitting algebra $(M,B)$ for $(L,A)$ and its Construction \ref{constmod} module.

\begin{lemma} \label{leftnil}  Let $A$ be an ideal of the Leibniz algebra $L$ such that $AL=0$.  Let $(M,B)$ be the Construction \ref{constsplit} splitting algebra for $(L,A)$ and let $V = \langle e, B \rangle$ be its Construction \ref{constmod} module. Let $x$ be an element of $L$ whose left action $d_x:L \to L$ is nilpotent.  Then the left action $\lambda_x: V \to V$ is nilpotent.
\end{lemma}

\begin{proof}  Since $B \simeq L$ as left $L$-module, $\lambda_x$ is nilpotent on $B$.  As $xV \subseteq B$, $\lambda_x: V \to V$ is nilpotent.
\end{proof}

We can consider similarly the elements $x\in L$ whose right action $r_x:L \to L$ given by $r_x(y) = yx$ is nilpotent.  As remarked earlier, these include the elements with nilpotent left action.  

\begin{lemma}  \label{rightnil} Let $A$ be an ideal of the Leibniz algebra $L$ such that $AL=0$.  Let $(M,B)$ be the Construction \ref{constsplit} splitting algebra for $(L,A)$ and let $V = \langle e, B \rangle$ be its Construction \ref{constmod} module. Then for all $x \in L$, the right action $\rho_x: V \to V$ is nilpotent.
\end{lemma}

\begin{proof}  For all $x \in L$, $Vx \subseteq B$ and $Bx = 0$.
\end{proof}

For a Lie algebra $L$, it was shown in Barnes \cite[Theorem 5.1]{extras}, that there exist modules with the additional property that for every saturated formation $\fF$ of soluble Lie algebras and every subnormal subalgebra $U \in \fF$ of $L$, $V$ is $(U, \fF)$-hypercentral.  I show that this holds for Leibniz algebras.  I require the following lemma.

\begin{lemma} \label{sn} Let $\fF$ be a saturated formation of soluble Leibniz algebras.  Suppose that $U \sn L$ and that $U \in \fF$.  Then $L$, regarded as $U$-module, is $(U,\fF)$-hypercentral. 

Let $A$ be an ideal of $L$ such that $AL=0$.  Let $(M,B)$ be the Construction \ref{constsplit} splitting algebra for $(L,A)$ and let $V = \langle e, B \rangle$ be its Construction \ref{constmod} module.  Then $V$ is $(U,\fF)$-hypercentral.
\end{lemma}

\begin{proof}  We have a chain of subalgebras $U = U_0 \ideq U_1 \ideq \dots \ideq U_n = L$.  Since $U$ acts trivially on each $U_i/ U_{i-1}$, $L/U$ is $(U,\fF)$-hypercentral.  Since $U \in \fF$, $U$ is $(U,\fF)$-hypercentral and it follows that $L$ is $(U,\fF)$-hypercentral.

By Barnes \cite[Theorem 3.16]{SchunckLeib}, it follows that $B$ is $(U,\fF$)-hypercentral.  As $V/B$ is the trivial module, it is $(U,\fF)$-central and so $V$ is $(U,\fF)$-hypercentral.
\end{proof}

\begin{theorem}  Let $L$ be a Leibniz algebra of dimension $n$.  Then there exists a faithful $L$-module $V$ of dimension less than or equal to $n+1$ and with the following additional properties:
\begin{enumerate}
\item The left action $\lambda_x$ of $x$ on $V$ is nilpotent for all $x \in L$ whose left action $d_x$ on $L$ is nilpotent.
\item The right action $\rho_x$ of $x$ on $V$ is nilpotent for all $x \in L$ whose right action $r_x$ on $L$ is nilpotent.
\item $V$ is $(U,\fF)$-hypercentral for every saturated formation $\fF$ of soluble Leibniz algebras and every $U \sn L$ with $U \in \fF$.
\end{enumerate}
\end{theorem}

\begin{proof}   If $\Z(L)=0$,  then $L$ as $L$-module satisfies all the conditions by Lemma \ref{sn}.  So suppose that $\Z(L) \ne 0$.  We take $A = \Ann(L)$ and the Construction \ref{constsplit} splitting algebra $(M,B)$ for $(L,A)$. Its Construction \ref{constmod} module $V$ has dimension $n+1$ and satisfies the other conditions by Lemmas \ref{leftnil}, \ref{rightnil} and  \ref{sn}. 
\end{proof}

\bibliographystyle{amsplain}

\end{document}